\documentclass{amsart}
\usepackage{amssymb}
\numberwithin{equation}{section}

\newcommand{\calA}{{\mathcal A}}

\newcommand{\calE}{{\mathcal E}}

\newcommand{\TAU}{\psi}
\newcommand{\PHI}{\varphi}
\newcommand{\XX}{{\mathbb X}}
\newcommand{\YY}{{\mathbb Y}}
\newcommand{\ZZ}{{\mathbb Z}}
\newcommand{\SSS}{{\mathbb S}}
\newcommand{\UU}{{\mathbb U}}
\newcommand{\VV}{{\mathbb V}}

\newcommand{\II}{{\mathbb I}} 

\newcommand{\sR}{{\mathbb R}}
\newcommand{\sC}{{\mathbb C}}
\newcommand{\sN}{{\mathbb N}}

\newcommand{\eps}{\varepsilon}
\newcommand{\la}{\lambda}

\newcommand{\citimes}{\circledast}

      \newtheorem{theorem}{Theorem}[section]
       \newtheorem{proposition}[theorem]{Proposition}
       \newtheorem{corollary}[theorem]{Corollary}
       \newtheorem{lemma}[theorem]{Lemma}
\newtheorem{claim}{Claim}
\theoremstyle{remark}
       \newtheorem{remark}{Remark}[section]
\theoremstyle{definition}
\newtheorem{definition}{Definition}[section]

\author{Marek Bo\.zejko}
\thanks{\noindent Research partially supported by  the Taft Research Center, KBN Grant No 1 PO3A 01330, and   NSF
grant \#DMS-0504198}
\address{Instytut Matematyczny\\ Uniwersytet Wroc{\l}awski\\ Pl. Grunwaldzki 2/4
50-384 Wroc{\l}aw, Poland.
}
\email{bozejko@math.uni.wroc.pl}
\author{
W{\l}odzimierz  Bryc
}
\address{
Department of Mathematics,
University of Cincinnati,
PO Box 210025,
Cincinnati, OH 45221--0025, USA}
\email{Wlodzimierz.Bryc@UC.edu}

\date{
Printed: \today.   File: {\tt\jobname.TEX}}

\title[Quadratic regression and CLT for two-state algebras]{A quadratic regression problem for two-state algebras with application
to the central limit theorem }
\begin{document}
\maketitle

\begin{abstract} We extend the free version \cite{Bozejko-Bryc-04}  of the Laha-Lukacs theorem
to probability spaces with two-states.
We then use this result to
generalize
the noncommutative central limit theorem  of Kargin \cite{Kargin-07} to the two-state
setting.
\end{abstract}

\section{Introduction}

Both classical and free Meixner distributions  first appeared in the theory of orthogonal polynomials in the works of Meixner  \cite{Meixner-40}, Anshelevich \cite{Anshelevich01}, and Saitoh and Yoshida \cite{Saitoh-Yoshida01}. Morris \cite{Morris82} pointed out the  relevance of classical Meixner distributions for the theory of exponential families in statistics;  Diaconis, Khare and Saloff-Coste \cite{Diaconis-08} gave an excellent overview of state of the art.  Ismail and May \cite{Ism:May} analyzed a mathematically equivalent problem from the point of view of approximation operators.  A counterpart of (some aspects of)  this theory for free Meixner distributions appear in an  unpublished manuscript by Bryc and Ismail \cite{Bryc-Ismail-05} and in  \cite{Bryc-06-08}.

Laha and Lukacs  \cite{Laha-Lukacs60} characterized all the (classical) Meixner distributions using a quadratic
regression property and Bo\.zejko and Bryc  \cite{Bozejko-Bryc-04} proved the corresponding free version. Anshelevich  \cite{anshelevich-2007}
considered
a Boolean version of this property showing that in the Boolean theory
Laha-Lukacs property characterizes only the Bernoulli distributions.

According to Example 3 in  \cite{anshelevich-2008a} and Proposition 3.1 of  Franz  \cite{Franz:2006}, Boolean, monotone, and free  independence are all special cases of the $c$-freeness for algebras with two states. Our primary goal in this paper is to extend   \cite{Bozejko-Bryc-04} and  \cite{anshelevich-2007} to the two-state setting
under a weaker form of $c$-freeness, which we call $(\PHI|\TAU)$-freeness, and which shares with boolean and free independence  a good description by cumulants.

As an application of our main result, we prove the central limit theorem under a certain type of ``weak dependence" which includes the so called singleton condition, whose importance to central limit theorem was pointed out in Theorem 0 of   Bo\.zejko and Speicher \cite{Bozejko-Speicher-96}; our assumptions are modeled on Kargin \cite{Kargin-07} who weakened freeness assumption in the free central limit theorem. Our result addresses a question  of finding the ``appropriate notions of {\em independence} or of {\em weak dependence}" for the quantum central limit theorem which was raised on page 11 of   \cite{AHO-98d} and describes the limit law; if one is interested solely in convergence, it can be deduced from  the general theory of the quantum  central limit theorem  developed by Accardi,  Hashimoto and Obata \cite{AHO-98c}. Section 8.2 of Hora and Obata \cite{Hora-Obata-07}   discusses   the role of singleton condition and gives the central limit theorem under classical, free, boolean, and monotone independence.


\subsection{A two-state freeness condition}
Let $\calA$ be a unital $*$-algebra with two states $\TAU, \PHI:\calA\to\sC$.
We assume that both states fulfill the usual assumptions of positivity and normalization,
and we assume tracial property $\TAU(ab)=\TAU(ba)$ for $\TAU$, but not for $\PHI$.

A typical model of an algebra with two sates is a  group algebra of a group $ G= *_i G_{i}$ , a free product of groups $ G_{i}$. Here
$\PHI$ is the boolean product of the individual states (which was also called "regular free state"); the simplest example is the free product of integers, $G_i=\ZZ$,  where $G$ is a free group with arbitrary number of generators, and $\PHI$ is the Haagerup state, $\Phi(x)=r^{|x|}$, where $|x|$ is the length of word $x\in G$, $-1\leq r\leq 1$, and state  $\TAU$ is $\delta(0)$.   For details see Bozejko \cite{Bozejko-86,Bozejko-87}.

A self-adjoint element $\XX\in\calA$
 with moments that fulfill appropriate  growth condition
 defines a pair $\mu,\nu$  of
 probability measures on $(\sR,\mathcal{B})$ such that
 $$
\PHI(\XX^k)=\int_\sR x^k\mu(dx) \mbox{ and }  \TAU(\XX^k)=\int_\sR x^k \nu(dx).
 $$
We will refer to measures  $\mu,\nu$ as the $\PHI$-law  and the $\TAU$-law of $\XX$, respectively.

With each set of $a_1,\dots,a_n\in\calA$ and a pair of states $(\PHI,\TAU)$ we associate the  cumulants
$R_k=R_{k,\PHI,\TAU}$, $k=1,2,\dots$, which are the multilinear functions  $\calA^k\to \sC$ defined by
\begin{multline}\label{R-cumulants}
  \PHI(a_1a_2\dots a_n)\\=\sum_{k=1}^n\sum_{1=s_1<s_2<\dots<s_k\leq n}R_{k}(a_1,a_{s_2},\dots,a_{s_k})
  \PHI(a_{s_{k}+1}\dots a_n)\prod_{r=1}^{k-1}\TAU\left(\prod_{j=s_r+1}^{s_{r+1}-1}a_j\right).
\end{multline}
 We will use the notation
\begin{equation}
  \label{r-cums}
  r_n(a_1,\dots,a_n):=R_{n,\TAU,\TAU}(a_1,\dots,a_n).
\end{equation}
We remark that $r_n$ are  the free cumulants with respect to state $\TAU$,
as defined by Speicher \cite{Speicher-94,Speicher-98}; see also  \cite{Nica-Speicher-06}.
For more general theory of cumulants, see  \cite{Lehner-04}.

  Fix $a\in\calA$, and consider the following formal power series
\begin{eqnarray}
 \label{RRR} R(z)&=&\sum_{n=1}^\infty R_n(a,\dots,a) z^{n-1},\\
  m(z)&=&\sum_{n=0}^\infty z^n \TAU(a^n),\\
  M(z)&=&\sum_{n=0}^\infty z^n \PHI(a^n).
\end{eqnarray}
  By Theorem 5.1 of  \cite{Bozejko-Leinert-Speicher}, Eqtn. \eqref{R-cumulants} is equivalent to
  the following relation
  \begin{equation}
    \label{MRm}
    M(z)\left(1-zR(zm(z))\right)=1.
  \end{equation}

\begin{definition}
  We say that subalgebras $\calA_1,\calA_2,\dots$ are $(\PHI|\TAU)$-free
   if for every choice of $a_1,\dots, a_n\in
\bigcup_j\calA_j$ we have
$$
R_n(a_1,\dots,a_n)=0 \mbox{ except if all $a_j$ come from the same algebra.}
$$
\end{definition}
It is important to note that  $(\PHI|\TAU)$-freeness is weaker than $c$-freeness, as explained before Lemma \ref{L1}. Thus  we could have
used the term {\em weak $c$-freeness} instead of $(\PHI|\TAU)$-freeness.

When the algebras are $(\TAU|\TAU)$-free, we will abbreviate this to $\TAU$-free.
From Ref.  \cite{Speicher-94} it follows that
 $\TAU$-freeness coincides with the usual concept of freeness as introduced by Voiculescu \cite{Voiculescu86}.

We will say that $\XX,\YY$ are $(\PHI|\TAU)$-free if the  unital algebras $\sC\langle \XX\rangle $
and $\sC\langle\YY\rangle$ are $(\PHI|\TAU)$-free.

A related concept is the following.
\begin{definition}[See Refs.  \cite{Bozejko-Speicher91} and  \cite{Bozejko-Leinert-Speicher}]
We say that subalgebras $\calA_1,\calA_2,\dots$ are $c$-free if for every
choice of $i_1\ne i_2\ne\dots\ne i_n$
and every choice of
$a_j\in\calA_j$ such that $\TAU(a_j)=0$ (thus $a_j\ne1$) we have
\begin{equation}\label{generalized free}
  \PHI(a_{i_1}\dots a_{i_n})=\prod_{k=1}^n \PHI(a_{i_k}).
\end{equation}
\end{definition}


\subsection{Properties of $(\PHI|\TAU)$-freeness}
  \label{rem R}
If $\calA_1, \calA_2$ are $(\PHI|\TAU)$-free then
  for $a\in\calA_1, b\in\calA_2$
  \begin{equation}
    \label{product2}
    \PHI(ab)=\PHI(a)\PHI(b).
  \end{equation}
  For $a_1,a_2\in\calA_1, b\in\calA_2$ we have
\begin{equation}
    \label{product3}
    \PHI(a_1ba_2)=\TAU(b)\PHI(a_1a_2)-\TAU(b)\PHI(a_1)\PHI(a_2)+\PHI(b)\PHI(a_1)\PHI(a_2).
  \end{equation}
    For $a_1,a_2\in\calA_1, b_1,b_2\in\calA_2$ we have
\begin{multline}
    \label{product4}
    \PHI(a_1b_1a_2b_2)=\PHI(a_1a_2)\TAU(b_1)\TAU(b_2)
-\PHI(a_1)\TAU(a_2)\PHI(b_1b_2)\\+\PHI(a_1)\TAU(a_2)\PHI(b_1)\PHI(b_2)-
\PHI(a_1)\PHI(a_2)\PHI(b_1)\PHI(b_2)
.
  \end{multline}

  Formulas \eqref{product2} \eqref{product3}
  and are identical to formulas under $c$-freeness as given in
 Lemma 2.1 of Ref.  \cite{Bozejko-Leinert-Speicher}. Together with formula \eqref{product4} they
  imply that for a pair of $(\PHI|\TAU)$-free algebras,  \eqref{generalized free} holds for $n\leq 4$.
One can check that if $a,b$ are $(\PHI|\TAU)$-free and
 $\TAU(a)=\TAU(b)=0$ but $\TAU(bab)\ne 0$ then
$\PHI(ababa)\ne \PHI(a)^3\PHI(b)^2$; thus the  concepts of $c$-freeness and of
$(\PHI|\TAU)$-freeness are not equivalent.
Nevertheless they coincide for $\TAU$-free algebras as noted in the following.
\begin{lemma}[page 368 of Ref.  \cite{Bozejko-Leinert-Speicher}]
  \label{L1}  Suppose $\calA_1,\calA_2,\dots$ are $\TAU$-free. Then
 the algebras $\calA_1,\calA_2,\dots$ are $(\PHI|\TAU)$-free if and only if they are $c$-free.
  \end{lemma}
%
(It would be interesting to characterize $(\PHI|\TAU)$-freeness without the freeness assumption on $\TAU$.)

We will also rely on the following fact.

\begin{lemma}[Ref  \cite{Bozejko-Leinert-Speicher}]
  \label{L2}  Given a noncommutative random variable $\XX$ in a two-state probability space,
  there exist a two-state algebra (which one can take as the algebra of noncommutative polynomials
  $\sC\langle \XX,\YY\rangle$ in two variables) and two non-commutative random variables $\widetilde{\XX},\widetilde{\YY}$ which are $\TAU$-free, $(\PHI|\TAU)$-free,
  and both have the same $\PHI$-law and $\TAU$-law as $\XX$.
  \end{lemma}
  \begin{proof}
  Theorem 1 of  Ref.   \cite{Bozejko-Speicher91}, see also
    Theorem 2.2 of Ref.  \cite{Bozejko-Leinert-Speicher}, shows how to extend both states  to
     the free product of the original algebra so that the resulting algebras are $c$-free and $\TAU$-free.
     By Lemma \ref{L1}, they are thus $(\PHI|\TAU)$-free.
  \end{proof}

\section{A $(\PHI|\TAU)$-free  quadratic regression problem}
In this section we prove a two-state version of Theorem 3.2 of Ref.  \cite{Bozejko-Bryc-04}.
The statement is fairly technical, but we found it useful
 for our proof of the central limit theorem (Theorem \ref{T2} below).
\begin{theorem}\label{T1}
  Suppose $\XX,\YY$ are self-adjoint $(\PHI|\TAU)$-free and
  \begin{equation}\label{equal laws}
    \PHI(\XX^n)=\PHI(\YY^n),\;  \TAU(\XX^n)=\TAU(\YY^n)
  \end{equation} for all $n$.
Furthermore,
assume that $\PHI(\XX)=0$, $\PHI(\XX^2)=1$.
(This can always be achieved by a shift and dilation, as long as $\PHI(\XX^2)\ne0$.)

  Let $\SSS=\XX+\YY$ and suppose that there are $a,c\in\sR$ and $b>-2$
  such that
  \begin{equation}\label{quadr var}
    \PHI\left((\XX-\YY)^2\SSS^n\right)=c\PHI\left((4\II+2a\SSS+b\SSS^2)\SSS^n\right), \;n=0,1,2\dots.
  \end{equation}
  Then
  the $\PHI$-moment generating functions $M_{\SSS}(z):=\sum_{k=0}^\infty z^k \PHI(\SSS^k)$ and
  $m_{\SSS}(z):=\sum_{k=0}^\infty z^k \TAU(\SSS^k)$, which are defined as formal power series,
  are related as follows
 \begin{equation} \label{MMM}
M_\SSS(z)=\frac{2+b-(2az+b)m_\SSS(z)}{2+b-(4z^2+2a z+ b)m_\SSS(z)}.
\end{equation}

\end{theorem}
\begin{remark}
  We will apply \eqref{MMM}
  to the case when $m_\SSS(z)$ converges for small enough $|z|$, in the form as written.
  In  general, the right hand side of \eqref{MMM} needs to be interpreted correctly.
  Recall that the composition $p(q(z))$ of two power series $p,q$
  is well defined if $q(z)$ has no constant term. Note that the formal power series
   $-b+(4z^2+2a z+ b)m_\SSS(z)$  has no constant term, so it can be composed
  with the formal power series
  $\sum_{n=0}^\infty\frac1{2^{n+1}}z^n$,
which is a formal power expansion of the function
 $\frac{1}{2-z}$. It is therefore natural to denote such a composition by
$$\frac{1}{2-(-b+(4z^2+2a z+ b)m_\SSS(z))}.$$
The right hand side of \eqref{MMM} is then interpreted as
  the product of this power series with the
   formal power series $2+b-(2az+b)m_\SSS(z)$.
\end{remark}

\begin{remark} Our assumptions on $\PHI$
do not allow us to use conditional expectations. However, it is still
natural to
ask which properties of conditional expectations would have implied
 assumptions of Theorem \ref{T1}.
To this end, we denote by $\PHI(\cdot|\SSS)$ the conditional expectation onto the
commutative algebra generated by $\SSS$.

From equality of the laws \eqref{equal laws} and $(\PHI|\TAU)$-freeness, one can deduce that
\begin{equation}
  \label{lin reg}
 \PHI(\XX\SSS^n)=\frac12\PHI\left(\SSS^{n+1}\right), \;n=0,1,2\dots.
\end{equation}
(See \eqref{EX} below.)  When the conditional expectation exists, this property follows from
$\PHI(\XX|\SSS)=\frac12\SSS$. We can then derive \eqref{quadr var} from the quadratic variance property
\begin{equation}\label{Var}
\PHI(\XX^2|\SSS)-(\PHI(\XX|\SSS))^2=c\left(\II+\frac{a}{2}\SSS+\frac{b}{4}\SSS^2\right).
\end{equation}
%
\end{remark}

\subsection{Proof of Theorem \ref{T1}}
 We first remark that $ c=(2+b)^{-1}$.
This follows from
 \eqref{quadr var} with $n=0$ since $\PHI(\XX\pm \YY)^2=2\pm\PHI(\XX\YY)\pm\PHI(\YY\XX)=2$.

By definition,  $R_n(\SSS,\dots,\SSS)=R_n(\XX,\dots,\XX)+R_n(\YY,\dots,\YY)$.
From \eqref{equal laws} we see that $R_n(\XX,\dots,\XX)=R_n(\YY,\dots,\YY)$. Thus
\begin{multline}\label{R1}
  R_n(\XX-\YY,\SSS,\dots,\SSS)=R_n(\XX,\SSS,\dots,\SSS)-R_n(\YY,\SSS,\dots,\SSS)\\
  =R_n(\XX,\dots,\XX)-R_n(\YY,\dots,\YY)=0
\end{multline}
for all $n$.
By \eqref{R-cumulants} this implies
\begin{equation}
  \label{EX}
  \PHI\left((\XX-\YY)\SSS^n\right)=0.
\end{equation}

Similarly, using multilinearity of $R$,
 \begin{multline}\label{R2}
  R_n(\XX-\YY,\XX-\YY,\SSS,\dots,\SSS)\\=
R_n(\XX,\XX-\YY,\SSS,\dots,\SSS)-R_n(\YY,\XX-\YY,\SSS,\dots,\SSS)\\
=R_n(\XX,\dots,\XX)+R_n(\YY,\dots,\YY)=
 R_n(\SSS,\dots,\SSS)
 \end{multline} for all $n\geq 2$.
Formula \eqref{R-cumulants} therefore implies that
\begin{multline*}
\PHI\left((\XX-\YY)^2\SSS^n\right)\\=
\sum_{k=2}^{n+2}\sum_{1=b_1<b_2=2<\dots<b_k\leq n+2}R_{k}(\XX-\YY,\XX-\YY,\SSS\dots,\SSS)
  \PHI(\SSS^{n-b_k-1})\prod_{r=1}^{k-1}\TAU(\SSS^{b_{r+1}-b_r-1})
  \\+
\sum_{k=1}^{n+2}  \sum_{1=b_1<2<b_2<\dots<b_k\leq n+2}R_{k}(\XX-\YY,\SSS,\dots,\SSS)
  \PHI(\SSS^{n-b_k-1})\prod_{r=1}^{k-1}\TAU(\SSS^{b_{r+1}-b_r-1}).
  \end{multline*}
  By \eqref{R1}, the second sum vanishes. Using \eqref{R2} we get
\begin{multline}
\label{C-moments}
\PHI\left((\XX-\YY)^2\SSS^n\right)\\
=\sum_{k=2}^{n+2}\sum_{1=b_1<b_2=2<\dots<b_k\leq n+2}R_{k}(\SSS,\SSS,\SSS\dots,\SSS)
  \PHI(\SSS^{n-b_k-1})\prod_{r=1}^{k-1}\TAU(\SSS^{b_{r+1}-b_r-1}).
\end{multline}
Comparing this with the decomposition for $\PHI(\SSS^{n+2})$ we see that
\begin{multline*}
\PHI\left((\XX-\YY)^2\SSS^n\right)=\PHI(\SSS^{n+2})\\
-\sum_{k=2}^{n+2}  \sum_{1=b_1<2<b_2<\dots<b_k\leq n+2}R_{k}(\SSS,\SSS,\dots,\SSS)
  \PHI(\SSS^{n-b_k-1})\prod_{r=1}^{k-1}\TAU(\SSS^{b_{r+1}-b_r-1}).
\end{multline*}
We now rewrite the last sum  based on the value of $m=b_2-b_1$, compare Ref.  \cite{Bozejko-Bryc-04}.
 We have
\begin{multline*}
\PHI\left((\XX-\YY)^2\SSS^n\right)=\PHI(\SSS^{n+2})\\
-\sum_{m=1}^{n}\TAU(\SSS^m)\sum_{k=2}^{n+2}  \sum_{1=b_1<1+m=b_2<\dots<b_k\leq n+2}R_{k}(\SSS,\SSS,\dots,\SSS)
  \PHI(\SSS^{n-b_k-1})\\ \times\prod_{r=1}^{k-1}\TAU(\SSS^{b_{r+1}-b_r-1}).
\end{multline*}
Since $b_2-b_1-1=m$, formula \eqref{R2} gives
\begin{multline*}
\sum_{k=2}^{n+2}  \sum_{1=b_1<1+m=b_2<\dots<b_k\leq n+2}R_{k}(\SSS,\SSS,\dots,\SSS)
  \PHI(\SSS^{n-b_k-1})\prod_{r=1}^{k-1}\TAU(\SSS^{b_{r+1}-b_r-1})\\
  =
  \sum_{k=2}^{n+2}  \sum_{1=b_1<1+m=b_2<\dots<b_k\leq n+2}R_{k}(\XX-\YY,\XX-\YY,\SSS,\dots,\SSS)
  \PHI(\SSS^{n-b_k-1})\TAU(\SSS^m)\\ \times\prod_{r=2}^{k-1}\TAU(\SSS^{b_{r+1}-b_r-1}).
\end{multline*}
Re-indexing the variables so that $b_2=2$ and inserting this into
\eqref{C-moments} we get
\begin{equation*}
\PHI\left((\XX-\YY)^2\SSS^n\right)=\PHI(\SSS^{n+2})-
\sum_{m=1}^n \TAU(\SSS^m)\PHI((\XX-\YY)^2\SSS^{n-m}).
\end{equation*}
Thus from \eqref{quadr var} we get
\begin{equation*}
  \PHI(\SSS^{n+2})=\frac{1}{2+b}\sum_{j=0}^n\TAU(\SSS^j)\left(4\PHI(\SSS^{n-j})+2 a \PHI(\SSS^{n-j+1})+b\PHI(\SSS^{n-j+2})\right).
\end{equation*}
A routine argument now relates the formal power series:
\begin{multline*}
  M_\SSS(z)=1+z^2\sum_{n=0}^\infty z^n \PHI(\SSS^{n+2})
  \\=1+\frac{z^2}{2+b}\sum_{n=0}^\infty \sum_{j=0}^nz^j\TAU(\SSS^j)z^{n-j}\left(4\PHI(\SSS^{n-j})+2 a \PHI(\SSS^{n-j+1})+b\PHI(\SSS^{n-j+2})\right)
 \\ =1+\frac{z^2}{2+b} \sum_{j=0}^\infty z^j\TAU(\SSS^j)\sum_{n=j}^\infty z^{n-j}\left(4\PHI(\SSS^{n-j})+2 a \PHI(\SSS^{n-j+1})+b\PHI(\SSS^{n-j+2})\right)
\\=
1+\frac{m_\SSS(z)}{2+b}\left(4z^2M_\SSS(z)+2 a z (M_\SSS(z)-1)  + b(M_\SSS(z)-1)\right).
\end{multline*}

\section{The $\PHI$-law of $\XX$}
In this section we are interested in one explicit case when Theorem \ref{T1}
allows us to determine the $\PHI$-law of $\XX$
from the $\TAU$-law of $\XX$. This case arises
when $\XX,\YY$ are $\TAU$-free and $(\PHI|\TAU)$-free with compactly supported laws. Then the $\PHI$-law and the $\TAU$-law
 of $\XX+\YY$ are determined uniquely
from the laws of $\XX,\YY$  by the
generalized convolution $\citimes$ which was introduced 
 by Bo\.zejko and Speicher \cite{Bozejko-Speicher91} and  studied in Refs.
  \cite{Bozejko-Leinert-Speicher,B-W97,B-W01,Krystek-Yoshida03,Krystek-Yoshida04}.
 The generalized convolution is a
 binary operation  on the pairs
 of compactly supported probability measures $(\mu,\nu)$.
 The analytic approach from Theorem 5.2 in Ref.  \cite{Bozejko-Leinert-Speicher} is especially convenient
 for explicit calculations.
According to this result, the generalized convolution
$(\mu_1,\nu_1)\citimes(\mu_2,\nu_2)$ of pairs of compactly supported probability measures
is a pair $(\mu,\nu)$ of
compactly supported probability measures which is determined by the following procedure.
Consider the Cauchy transforms
\[ G_j(z)=\int\frac{1}{z-x}\mu_j(dx), \; g_j(z)=\int\frac{1}{z-x}\nu_j(dx), \; j=1,2.\]
  Let $k_j(z)$ be the inverse function of $g_j(z)$ in a
neighborhood of $\infty$, and define
\begin{equation}\label{EQ: r} r_j(z)=k_j(z)-1/z.\end{equation}
On the second component the $c$-convolution acts
as the free convolution \cite{Voiculescu86}, $\nu=\nu_1\boxplus\nu_2$.
Recall that the free convolution $\nu$ of measures $\nu_1,\nu_2$ is
 the unique probability measure
with the Cauchy transform $g(z)$ which solves the equation
\[g(z)=\frac{1}{z-r_1(g(z))-r_2(g(z))}.\]

To define the action of the generalized convolution on the first component, let
\[R_j(z)=k_j(z)-1/G_j(k_j(z)).\]
Thus
\begin{equation}
  \label{R2G}
G_j(z)=\frac{1}{z-R_j(g_j(z))}.
\end{equation}
The first component of the generalized convolution is defined as the unique
 probability measure  $\mu$ with the Cauchy transform
\[
G(z)=\frac{1}{z-R_1(g(z))-R_2(g(z))}\;.
\]
 We
write \[ (\mu,\nu)=(\mu_1,\nu_1)\citimes(\mu_2,\nu_2).
\]

We remark that
$$r(z)=\sum_{k=1}^\infty r_k z^{k-1}, \;\; R(z)=\sum_{k=1}^\infty R_k z^{k-1}$$
are the generating functions for the $\TAU$-free and $(\PHI|\TAU)$-free cumulants respectively, see \eqref{RRR}.
We also note that the above relations can be interpreted as combinatorial relations between $\TAU$-moments and
$\PHI$-moments; the assumption of compact support allows us to determine the laws uniquely from moments.
\subsection{The case of ``constant conditional variance"}

\begin{proposition}
  \label{L3} Suppose $\XX,\YY$ are $\TAU$-free with the same compactly supported
  $\TAU$-law $\nu$, and are $(\PHI|\TAU)$-free with the same $\PHI$-law.
  If \eqref{quadr var} holds with $a=b=0$, then the $\PHI$-law of $\XX$ is compactly supported
  and uniquely determined by $\nu$.
\end{proposition}
\begin{proof}
The $\TAU$-law of $\SSS$ is the free convolution $\nu\boxplus\nu$, so it is compactly supported.
Therefore $m_\SSS(z)$ is given by a series that converges for small enough $|z|$.
 Then \eqref{MMM} reduces to
\begin{equation*}
  M_\SSS(z)=\frac{1}{1-2 z^2 m_\SSS(z)},
\end{equation*}
and $M_\SSS(z)$ is also given by a convergent series. In particular, the $\PHI$-law
of $\SSS$ is compactly supported.
So for $\Im z>0$, the Cauchy transform is
\begin{equation}\label{g2G}
  G_\SSS(z)=\frac1zM_\SSS(1/z)=\frac{1}{z-2 g_\SSS(z)}.
\end{equation}
Thus $R_k(\SSS,\dots,\SSS)=0$ for all $k$ except for $R_2(\SSS,\SSS)=2$.
This shows that $R_k(\XX,\dots,\XX)=0$ for all $k$ except for $R_2(\XX,\XX)=1$.
Thus $R_\XX(z)=z$ and \eqref{MRm} gives
\begin{equation}
 \label{M-x} M_\XX(z)=\frac{1}{1- z^2 m_\XX(z)}.
\end{equation}
This implies that $\PHI$-law of $\XX$ has compact support,
and its Cauchy transform is uniquely determined by
\begin{equation}
 \label{G-x} G_\XX(z)=\frac{1}{z- g_\XX(z)}.
\end{equation}
%
\end{proof}

In particular, suppose $\nu$ is
the semicircle law with mean zero and variance $\sigma^2$, so that
 $g_\XX(z)=\frac{z-\sqrt{z^2-4\sigma ^2}}{2\sigma ^2}$.
 Proposition \ref{L3} then
 shows that the $\PHI$-law of $\XX$ has Cauchy-Stieltjes transform
\begin{equation}\label{G_X}
G_\XX(z)=\frac{(\sigma ^2-\frac12)\,z -\frac12{\sqrt{z^2-4\,\sigma ^2}}}
  {1+\left( \sigma ^2-1 \right) \,z^2}.
\end{equation}
This law plays the role of the ``Gaussian limit" in  Ref.  \cite{Bozejko-Leinert-Speicher}.
%
\subsection{The case of ``linear conditional variance"}
Suppose \eqref{quadr var} holds with $b=0$. Then \eqref{MMM} reduces to
\begin{equation*}
  M_\SSS(z)=\frac{1-a z m_\SSS(z)}{1-(2z+a)z m_\SSS(z)}.
\end{equation*}
So again the $\Phi$-law of $\SSS$ is compactly supported, if the $\TAU$-law is, and
the Cauchy transform is
\begin{equation*}
  G_\SSS(z)=\frac{1-a g_\SSS(z)}{z-(2+az)g_\SSS(z)}=\frac{1}{z-R_\SSS(g_\SSS(z))}
\end{equation*}
with
\begin{equation*}
  R_\SSS(u)=\frac{2u}{1-au}.
\end{equation*}
This shows that $R_\XX(z)=\frac{z}{1-az}$ and
\begin{equation*}
  G_\XX(z)=\frac{1-a g_\XX(z)}{z-(1+az)g_\XX(z)}.
\end{equation*}
In particular, suppose that the $\TAU$-law of $\XX$ is Marchenko-Pastur with parameter $\la>0$, so that
$$g_\XX(z)=\frac{z+(1-\la)-\sqrt{(z-1-\la)^2-4\la}}{2z}.$$
If $a=1$, then the $\PHI$-law of $\XX$ is compactly supported, with Cauchy transform
$$
  G_\XX(z)=\frac{1 + \lambda
    -z ( 1 -
       2 \lambda  )  -
    {\sqrt{
        {( z-1 -
           \lambda  ) }^2-4 \lambda  }}}
    {2\left(1
     +  z ( 1 +
       \lambda  )-  z^2
    ( 1 - \lambda  )\right) }.
$$
Related laws appear in Eqtn. (17) of Ref.  \cite{Bryc-Wesolowski-04} and
 on page 380 in Ref.  \cite{Bozejko-Leinert-Speicher}.

\section{Central limit theorem for non-identical summands}
The central limit theorem  and the Poisson convergence theorem for sums of
$(\PHI|\TAU)$-free random variables that are also $\TAU$-free
 appear in Theorems 4.3 and 4.4 of Ref.  \cite{Bozejko-Leinert-Speicher}. Recently
 Kargin \cite{Kargin-07} observed that in the free case one can dispense
 with the assumption of identical laws and at the same time relax the freeness assumption. A similar   result in classical probability is due to
 Komlos \cite{komlos1973clt} who assumes a much weaker version of singleton condition \eqref{singleton1}  and has an inequality  in his condition (6) that substitutes for \eqref{Kargin 2}.
Komlos' conditions were motivated by   (classical) central limit theorem for the so called multiplicative systems. We also note that in classical probability  Jakubowski and Kwapie\'n \cite{Jakubowski-Kwapien-79} discovered a beautiful connection between multiplicative systems and independent sequences. No counterpart of this result is known in noncommutative setting; compare also  non-commutative $p$-orthogonality and Remark 2.4 of Pisier \cite{Pisier:2000}, and work of K\"ostler and  Speicher \cite{Kostler-Speicher:2008} on noncommutative versions of de~Finetti's theorem.

 In this section we use Theorem \ref{T1} to deduce a 
 two-state version of Kargin's result.
The convergence of moments can also be obtained  as a corollary of Theorem 3 in Accardi Hashimoto and Obata \cite{AHO-98d}, see also  Theorem 3.3 in  \cite{AHO-98c}, Theorem 0 of  \cite{Bozejko-Speicher-96}, and Section 8.2 in  \cite{Hora-Obata-07}.
This theorem says that under the singleton condition \eqref{singleton1}, in order to
complete
the proof of CLT, it suffices to control ergodic averages
of totally entangled pair partitions. The disentanglement can be
achieved from various conditions that include   statistical conditions, such as   the free
case or the generalized freeness given by   conditions \eqref{singleton2} and \eqref{Kargin 2}. This approach, as well as  classical CLT in Ref~ \cite{komlos1973clt}, suggests that one should seek
a weaker version of \eqref{Kargin 2} that perhaps would be stated as an inequality.  On the other hand, our proof from Theorem \ref{T1} gives directly the formula for the Cauchy-Stieltjes transform of the limit law which would require additional work if the techniques from  \cite{AHO-98c} were applied.

 We also note that Wang \cite{Wang-08} uses analytical methods to study limit theorems for additive $c$-convolution with measures of unbounded support.
  It is not obvious how Kargin's condition A should be generalized to this setting.  In fact, a generalization of  Theorem \ref{T1}  to   unbounded random variables   would be interesting even in the free case studied in  \cite{Bozejko-Bryc-04}.

 \begin{definition}
   We will say that a sequence of random variables $\XX_1,\XX_2,\dots$ satisfies
    Kargin's Condition A with respect to  $(\PHI|\TAU)$, if:
       \begin{enumerate}
     \item For every $k\not\in\{j_1,\dots, j_n\}$ the following singleton conditions hold:
     \begin{multline}
       \label{singleton1}
            \PHI(\XX_k\XX_{j_1}\dots \XX_{j_r})=\PHI(\XX_{j_1}\XX_k\XX_{j_2}\dots \XX_{j_r})=\dots\\=\PHI(\XX_{j_1}\dots \XX_{j_r}\XX_k)=0.
     \end{multline}
     \begin{equation}
       \label{singleton2}
            \TAU(\XX_k\XX_{j_1}\dots \XX_{j_r})=0.
     \end{equation}
     (In particular, $\TAU(\XX_j)=\PHI(\XX_j)=0$.)
\item For every $k\not\in{\{j_1,\dots, j_r\}}$,  and $0\leq p\leq r$,
\begin{multline}
  \label{Kargin 2}
\PHI(\XX_k\XX_{j_1}\dots \XX_{j_p}\XX_k\XX_{j_{p+1}}\dots \XX_{j_r})\\=
\PHI(\XX_k^2)\TAU(\XX_{j_1}\dots \XX_{j_p})\PHI(\XX_{j_{p+1}}\dots \XX_{j_r}).
\end{multline}
   \end{enumerate}
 \end{definition}

We remark that conditions \eqref{singleton1} and \eqref{Kargin 2} are automatically satisfied if $\XX_1,\XX_2,\dots$
are $\PHI$-centered and $(\PHI|\TAU)$-free; clearly,
condition \eqref{singleton2} holds true if $\XX_1,\XX_2,\dots$ are $\TAU$-centered
and $\TAU$-free but of course it is weaker and can hold also for classical (commutative)
independent random variables.

 \begin{theorem}
   \label{T2}
   Suppose that \begin{enumerate}
    \item $\XX_1,\XX_2,\dots$ satisfies  Kargin's Condition A with respect to $(\PHI|\TAU)$;
     \item All joint moments of order $k$ are uniformly bounded
  \begin{equation}
    \label{Unif Bound}
     \sup_{j_1,\dots,j_k\geq 1}|\PHI(\XX_{j_1}\dots \XX_{j_k})|\leq C_k<\infty \mbox{ for $k=1,2,\dots$}.
  \end{equation}
   \item Sequences $s_j^2:=\TAU(\XX_j^2)$ and $S_j^2:=\PHI(\XX_j^2)$ satisfy
   \begin{equation}
     \label{S-limit}
     (s_1^2+\dots+s_n^2)/n\to s \mbox{ and } (S_1^2+\dots+S_n^2)/n\to S.
   \end{equation}
   \item $0<s,S<\infty$.
    \item The $\TAU$-moments of $\frac1{\sqrt{s_1^2+\dots+s_n^2}}\sum_{j=1}^n\XX_j$
    converge to the corresponding moments
    of a compactly supported
    probability measure
    $\nu$.
   \end{enumerate}
Then
the $\PHI$-moments of $\frac1{\sqrt{S_1^2+\dots+S_n^2}}\sum_{j=1}^n\XX_j$ converges to the moments of
the unique compactly supported law $\mu$ with 
Cauchy transform \eqref{G-x}, where
$g_\XX(z)=\int \frac{S}{Sz-s x }\nu(dx)$. \end{theorem}

 Combining Theorem \ref{T2} with Ref.  \cite{Kargin-07} and formula \eqref{G_X}
  we get the following generalization of Theorem 4.3 in Ref.  \cite{Bozejko-Leinert-Speicher}.
\begin{corollary}
  \label{C-kargin}
Suppose that \begin{enumerate}
    \item $\XX_1,\XX_2,\dots$ satisfies  Kargin's Condition A with respect to $(\PHI|\TAU)$
    and with respect to $(\TAU|\TAU)$.
     \item  All moments are uniformly bounded: \eqref{Unif Bound} holds true, and
   $\sup_n|\TAU(\XX_n^k)|<\infty$ for $k=1,2,\dots$.
   \item Sequence $s_j^2:=\TAU(\XX_j^2)=s_j^2$ and $S_j^2:=\PHI(\XX_j^2)$ satisfy \eqref{S-limit} with
    $0<s,S<\infty$.
    \end{enumerate}
Then
the $\PHI$-law of $\frac1{\sqrt{S_1^2+\dots+S_n^2}}\sum_{j=1}^n\XX_j$ converges to the law
$
\mu
$ with the Cauchy-Stieltjes transform \eqref{G_X} and $\sigma=s/S$.
\end{corollary}

Our   proof of the central limit theorem is based on reduction to Laha-Lukacs theorem which  in classical probability  was introduced in Section 7.3.1  of Bryc \cite{Bryc95}.
\subsection{Proof of Theorem \ref{T2}}
By Ref.  \cite{Bozejko-Speicher91} without loss of generality we may assume that we  have a two-state probability space with
 two copies of the original sequence: $(\XX_k)$ and $(\YY_k)$ each of them separately having the same
 $\TAU$-moments and $\PHI$-moments as the original sequence, but such that
 the algebras $\calA_\XX$ and $\calA_\YY$ generated by $(\XX_k)$ and by $(\YY_k)$, respectively,
 are $\TAU$-free and $(\PHI|\TAU)$-free.

 Under this representation, the $\TAU$-distribution of $\frac1{\sqrt{s_1^2+\dots+s_n^2}}\sum_{j=1}^n(\XX_j+\YY_j)$
 converges to $\nu\boxplus\nu$. Our goal is to show that the $\PHI$-distribution of
 $\frac1{\sqrt{S_1^2+\dots+S_n^2}}\sum_{j=1}^n(\XX_j+\YY_j)$ has the unique limit determined by the law
 with Cauchy-Stieltjes transform \eqref{g2G}.
 To do so, denote
 $$\UU_n=\frac1{\sqrt{n}}\sum_{j=1}^n\XX_j,\; \VV_n=\frac1{\sqrt{n}}\sum_{j=1}^n\YY_j,\;
\SSS_n=\UU_n+\VV_n.$$
Denote
 $$\ZZ_j^{(\eps)}=\XX_j^\eps \YY_j^{1-\eps}, \; \eps=0,1.$$
 Since the variables do not commute, we adopt a special convention for the
 product notation
 convention which  relies on the order of the index set:
 $$\PHI\left(\prod_{s=1}^p \ZZ_{J(s)}^{\eps(s)}\right):=
 \PHI\left( \ZZ_{J(1)}^{\eps(1)}\ZZ_{J(2)}^{\eps(2)}\dots \ZZ_{J(p)}^{\eps(p)}\right).$$
\begin{lemma}
  \label{L4.5} In the above setting, if  $\{\XX_j\}$ satisfies Kargin's Condition A, then
  $\{\XX_1,\YY_1,\XX_2,\YY_2,\dots\}$  satisfies Kargin's Condition A.
\end{lemma}
\begin{proof} We first note the following.
\begin{claim}
  \label{Claim1}
  Singleton properties \eqref{singleton1}, \eqref{singleton2}
 for $\{\XX_j\}$ are equivalent to the following:
for $k\not\in\{j_1,\dots,j_p\}$ with $p=0,1,2,\dots$, we have
\begin{multline}\label{&}
  R_{p+1}(\XX_k,\XX_{j_1},\dots,\XX_{j_p})=R_{p+1}(\XX_{j_1}\XX_k,\XX_{j_2},\dots,\XX_{j_p})=\dots\\
  \dots=
R_{p+1}(\XX_{j_1},\dots,\XX_{j_p},\XX_k)=0,
\end{multline}
and
\begin{equation}\label{&&}
  r_{p+1}(\XX_k,\XX_{j_1},\dots,\XX_{j_p})=0.
\end{equation}
\end{claim}
\begin{proof}
  Clearly, \eqref{&&} implies \eqref{singleton2} by \eqref{R-cumulants} applied to $\PHI=\TAU$. Conversely,
  suppose that $r_{p+1}(\XX_k,\XX_{j_1},\dots,\XX_{j_p})\ne0$ for some $p\geq 0$, and take the smallest $p$.
Since for $F=\{f_1,f_2,\dots\}\subset\{j_1,\dots,j_p\}$,
$$
\TAU(\XX_k\prod_{f\in F}\XX_f)=0,
$$
the only non-zero terms in \eqref{R-cumulants} must come from cumulants that have $\XX_k$ as their argument.
Thus, with $\Pi_F$ denoting the appropriate products of moments,
\begin{multline*}
  0=\TAU(\XX_k,\XX_{j_1}\dots\XX_{j_p})=\sum_{F}r_{|F|+1}(\XX_k,\XX_{f_1},\XX_{f_2},\dots)\Pi_F\\=
r_{p+1}(\XX_k,\XX_{j_1},\dots,\XX_{j_p})+\mbox{lower order terms}.
\end{multline*}
Since by assumption all lower order cumulants vanish, we see that
$r_{p+1}(\XX_k,\XX_{j_1},\dots,\XX_{j_p})$ in fact
must be zero.
\end{proof}
\begin{claim}
  \label{Claim2}
  Suppose $\{\XX_j\}$ satisfies singleton properties \eqref{singleton1} and \eqref{singleton2}.
 Then \eqref{Kargin 2} is equivalent to the following:
for $k\not\in\{j_1,\dots,j_r\}$ with $r=1,2,\dots$, and every $0\leq p\leq r$ we have
\begin{equation}\label{&&&}
  R_{r+2}(\XX_k,\XX_{j_1},\dots,\XX_{j_p},\XX_k,\XX_{j_{p+1}}\dots,\XX_{j_r})=0.
\end{equation}
\end{claim}
\begin{proof}
Suppose \eqref{&} and \eqref{&&&} hold. Then in \eqref{R-cumulants}, $\XX_k$ must appear twice in
the argument of $R$. Thus
\begin{multline*}
\PHI(\XX_k\XX_{j_1}\dots \XX_{j_p}\XX_k\XX_{j_{p+1}}\dots \XX_{j_r})
\\=
R_2(\XX_k,\XX_k)\TAU(\XX_{j_1}\dots \XX_{j_p})\PHI(\XX_{j_{p+1}}\dots \XX_{j_r})+
\mbox{sum involving higher cumulants}\\
=
\PHI(\XX_k^2)\TAU(\XX_{j_1}\dots \XX_{j_p})\PHI(\XX_{j_{p+1}}\dots \XX_{j_r})+0.
\end{multline*}
Conversely, suppose that
$ R_{r+2}(\XX_k,\XX_{j_1},\dots,\XX_{j_p},\XX_k,\XX_{j_{p+1}}\dots,\XX_{j_r})\ne 0$,\ for some $r\geq 1$,
and take the smallest such $r$.
By \eqref{&}, expansion \eqref{R-cumulants} has no singleton appearances of $\XX_k$.
Thus
\begin{multline*}
\PHI(\XX_k\XX_{j_1}\dots \XX_{j_p}\XX_k\XX_{j_{p+1}}\dots \XX_{j_r})=  R_{r+2}(\XX_k,\XX_{j_1},\dots,\XX_{j_p},\XX_k,\XX_{j_{p+1}}\dots,\XX_{j_r})
\\+\sum_{\alpha=0}^{r-1}
\sum_{\#F=\alpha} R_{\alpha+2}(\XX_k,\XX_{f_1},\dots,\XX_{f_a},\XX_k,\XX_{f_{a+1}},\dots)
\\
=R_{r+2}(\XX_k,\XX_{j_1},\dots,\XX_{j_p},\XX_k,\XX_{j_{p+1}}\dots,\XX_{j_r})\\+
R_2(\XX_k,\XX_k)\TAU(\XX_{j_1}\dots \XX_{j_p})\PHI(\XX_{j_{p+1}}\dots \XX_{j_r}).
\end{multline*}
Thus $R_{r+2}(\XX_k,\XX_{j_1},\dots,\XX_{j_p},\XX_k,\XX_{j_{p+1}}\dots,\XX_{j_r})=0$.

\end{proof}

We will show that $\{\ZZ_j^{\eps(j)}\}$ satisfies Kargin's Condition A for any choice of indices
$(j,\eps(j))\in\sN\times\{0,1\}$.
Since the assumptions are symmetric with respect to $\{\XX_j\}$ and $\{\YY_j\}$, it is enough to analyze
the case when the distinguished element is $\XX_k=\ZZ_k^{(1)}$.

Suppose  $(1,k)\not\in\{(\eps(1),j_1),(\eps(2),j_2),\dots (\eps(p),j_p)\}$.
Then
$$r_{p+1}(\XX_k,\ZZ_{j_1}^{\eps(1)},\dots,\ZZ_{j_p}^{\eps(p)})=0.$$
Indeed, this holds true by $\TAU$-freeness of $\calA_\XX,\calA_\YY$ if one of the $\eps(i)=0$. On the other hand,
if all $\eps(i)=1$, then this holds true by \eqref{&}.
Similarly,
 $(\PHI|\TAU)$-freeness of $\calA_\XX,\calA_\YY$  implies that
  $$R_{p+1}(\XX_k,\ZZ_{j_1}^{\eps(1)},\dots,\ZZ_{j_p}^{\eps(p)})=0$$
either because some of the $\eps(j)=0$, or by \eqref{&}.
Thus \eqref{singleton1} and \eqref{singleton2} hold for $\{\ZZ_j^{\eps(j)}\}$
 $\{\ZZ_j^{\eps(j)}\}$ by
 Claim \ref{Claim1}.

Similarly, if $r\geq 1$,
$$
R_{r+2}(\XX_k,\ZZ_{j_1}^{\eps(1)},\dots,\ZZ_{j_p}^{\eps(1+p)},\XX_k,\ZZ_{j_{p+1}}^{\eps(p+1)}\dots,\ZZ_{j_r}^{\eps(r)})=0$$
either because some of $\eps(i)=0$ and  $\calA_\XX,\calA_\YY$  are
 $(\PHI|\TAU)$-free,
 or by \eqref{&&&}. Therefore
\eqref{Kargin 2} holds for $\{\ZZ_j^{\eps(j)}\}$ by Claim \eqref{Claim2}.
\end{proof}

\begin{lemma}\label{L5} For fixed $j,k,m\geq 0$,
  $$\sup_n |\PHI\left(\UU_n^j\VV_n^k(\UU_n+\VV_n)^m\right)|<\infty.$$
\end{lemma}
\begin{proof}
 Expanding the product, by Lemma \ref{L4.5} we see that
\begin{multline}\label{expand}
  \PHI\left(\UU_n^j\VV_n^k(\UU_n+\VV_n)^m\right)\\
  =n^{-(j+k+m)/2}\sum_{J:\{1,\dots,j+k+m\}\to\{1,\dots,n\}}
  \sum_{\eps\in\calE}
  \PHI\left(\prod_{s=1}^{j+k+m} \ZZ_{J(s)}^{\eps(s)}\right)
  \\
  =n^{-(j+k+m)/2}\sum_{J\in \mathcal{J}_{\geq 2}}\sum_{\eps\in\calE}
  \PHI\left(\prod_{s=1}^{j+k+m} \ZZ_{J(s)}^{\eps(s)}\right),
\end{multline}
where
$$
\mathcal{J}_{\geq 2}=\{J: \# J^{-1}(s)\ne 1  \mbox{ for all $1\leq s\leq n$}\}
$$
is the set of mappings $J:\{1,\dots,j+k+m\}\to\{1,\dots,n\}$ that take no singleton values,
and
\begin{multline*}
\calE=\{\eps\in 2^{\{1,\dots,j+k+m\}}: \\
\eps(1)=\dots=\eps(j)=1, \; \eps(j+1)=\dots=\eps(j+k)=0\}.
\end{multline*}

The cardinality of the first set can be bounded above by $\#\mathcal{J}_{\geq 2}\leq n^{(j+k+m)/2}$,  and
$\#\calE=2^m$, so by \eqref{Unif Bound},
$$\left|\sum_{J\in \mathcal{J}_{\geq 2}}\sum_{\eps\in \calE}
  \PHI\left(\prod_{s=1}^{j+k+m} \ZZ_{J(s)}^{\eps(s)}\right)\right|\leq C_{j+k+m} 2^m n^{(j+k+m)/2}.$$
\end{proof}
  Let
$$
\mathcal{J}_{2}=\{J: \# J^{-1}(s)=0,2  \mbox{ for all $1\leq s\leq n$}\}
$$
be the subset of $\mathcal{J}_{\geq 2}$ that consists of all mappings $J:\{1,\dots,j+k+m\}\to\{1,\dots,n\}$
that are two-to-one valued.
(Clearly $\mathcal{J}_{2}=\emptyset$ when $j+k+m$ is odd.)

\begin{lemma}
  \label{L5.5} For $j,k,m\geq 0$,
\begin{equation}
  \label{compare}
  \limsup_{n\to\infty} \left|\PHI(\UU_n^j\VV_n^k(\UU_n+\VV_n)^m) -
n^{-(j+k+m)/2}\sum_{J\in \mathcal{J}_2}\sum_{\eps\in \calE}
  \PHI(\prod_{s=1}^{j+k+m} \ZZ_{J(s)}^{\eps(s)})\right|=0,
\end{equation}
\end{lemma}
\begin{proof}
If there is a value $s\in\{1\dots n\}$ that is taken by $J$ at three or more different points,
then there are at most $j+k+m-1$ points on which $J$ is two-to-one. Therefore,
 $$\# \left(\mathcal{J}_{\geq 2}\setminus \mathcal{J}_2\right)\leq \left(\begin{matrix}
 j+k+m\\3\end{matrix}\right)n^{(j+k+m-1)/2},$$
 and by  \eqref{Unif Bound},
 the result follows from
 \eqref{expand},

  \begin{equation*}
  \sum_{J\in \mathcal{J}_{\geq 2}\setminus \mathcal{J}_2}\sum_{\eps }
  \left|\PHI(\prod_{s=1}^{j+k+m} \ZZ_{J(s)}^{\eps(s)})\right|\leq \left(\begin{matrix}
 j+k+m\\3\end{matrix}\right)C_{j+k+m} 2^m  n^{(j+k+m-1)/2}.
\end{equation*}

\end{proof}
We remark that since $\mathcal{J}_2=\emptyset$ for odd $j+k+m$, Lemma \ref{L5.5} implies that
$$
\limsup_{n\to\infty} \left|\PHI((\UU_n+\VV_n)^m)\right|=0 \mbox{ for odd $m$}.
$$

The next lemma is the main tool in identifying the limit via Theorem \ref{T1}.
\begin{lemma}\label{L6} For $m\geq 1$,
  $$\limsup_{n\to\infty}\left|\PHI((\UU_n-\VV_n)^2\SSS_n^m)-
  2\PHI(\SSS_n^{m}){\sum_{j=1}^nS_j^2/n}\right|=0.$$
\end{lemma}
\begin{proof}
  Since $(x-y)^2=x(x-y)+y(y-x)$, and the joint moments of $(\UU_n,\VV_n)$ are symmetric in $\UU_n,\VV_n$,
  it is enough to show that
   \begin{equation}
     \label{*}
     \limsup_{n\to\infty}\left|\PHI\left(\UU_n(\UU_n-\VV_n)\SSS_n^m\right)-
   \PHI\left(\SSS_n^{m}\right)\sum_{j=1}^nS_j^2/n\right|=0.
   \end{equation}
   By Lemma \ref{L5.5}, once we expand the sums in $\PHI\left(\UU_n^2\SSS_n^m -\VV_n\SSS_n^m-
\SSS_n^{m}\sum_{j=1}^nS_j^2/n\right)$, the only contributing terms come from the sum over
  the  two-to-one functions $J:\{1\dots m+2\}\to\{1\dots n\}$.
Therefore, it is enough to show that   before taking the limit, we have the following identity:
\begin{multline}\label{3.17}
 \frac1n\sum_{J\in\mathcal{J}_2}\sum_{\eps} \Big(  \PHI\left(\XX_{J(1)}\XX_{J(2)}\prod_{s=3}^{m+2}\ZZ_{J(s)}^{\eps(s)}\right)
 -\PHI\left(\XX_{J(1)}\YY_{J(2)}\prod_{s=3}^{m+2}\ZZ_{J(s)}^{\eps(s)}\right)\\
 -S_{J(1)}^2\delta_{J(1)=J(2)}\PHI\left(\prod_{s=3}^{m+2}\ZZ_{J(s)}^{\eps(s)}\right)\Big)=0.
\end{multline}
   Let $\mathcal{J}_*\subset\mathcal{J}_2$ denote the set of two-to-one
  functions with $J(1)=J(2)$. Expanding the products we see that for $J\in \mathcal{J}_*$ each term in \eqref{3.17}
  can be written as
  $$
 \PHI\left(\XX_{J(1)}\XX_{J(2)}\prod_{s=3}^{m+2}\ZZ_{J(s)}^{\eps(s)}\right)
  -\PHI\left(\XX_{J(1)}\YY_{J(2)}\prod_{s=3}^{m+2}\ZZ_{J(s)}^{\eps(s)}\right)
  -S_{J(1)}^2\PHI\left(\prod_{s=3}^{m+2}\ZZ_{J(s)}^{\eps(s)}\right).
  $$
Since $\YY_{J(2)}$ is a singleton, by Lemma
\ref{L4.5}, $\PHI\left(\XX_{J(1)}\YY_{J(2)}\prod_{s=3}^{m+2}\ZZ_{J(s+2)}^{\eps(s)}\right)=0$. The same lemma
gives
$$\PHI\left(\XX_{J(1)}\XX_{J(2)}\prod_{s=3}^{m+2}\ZZ_{J(s)}^{\eps(s)}\right)=\PHI\left(\XX_{J(1)}^2\prod_{s=3}^{m+2}\ZZ_{J(s)}^{\eps(s)}\right)
=S_{J(1)}^2\PHI\left(\prod_{s=3}^{m+2}\ZZ_{J(s)}^{\eps(s)}\right).$$
Thus
\begin{multline}\label{tmp*}
 \sum_{J\in\mathcal{J}_*}\sum_{\eps}\left(\PHI\left(\XX_{J(1)}\XX_{J(2)}\prod_{s=3}^{m+2}\ZZ_{J(s)}^{\eps(s)}\right)
  -\PHI\left(\XX_{J(1)}\YY_{J(2)}\prod_{s=3}^{m+2}\ZZ_{J(s)}^{\eps(s)}\right)\right)\\=
   \PHI\left(\SSS_n^{m}\right)\sum_{j=1}^nS_j^2.
\end{multline}
To end the proof, we need to show that the sum over $J\in\mathcal{J}_2\setminus\mathcal{J}_*$ is zero.
In fact, we observe that for each $J\in \mathcal{J}_2\setminus \mathcal{J}_*$,
\begin{equation}\label{tmp**}
\sum_{\eps}\left(\PHI\left(\XX_{J(1)}\XX_{J(2)}\prod_{s=3}^{m+2}\ZZ_{J(s)}^{\eps(s)}\right)
  -\PHI\left(\XX_{J(1)}\YY_{J(2)}\prod_{s=3}^{m+2}\ZZ_{J(s)}^{\eps(s)}\right)\right)=0.
\end{equation}
To see this, denote by $r>2$ the unique index with $J(1)=J(r)$. Given $\eps\in 2^{\{3,\dots,m+2\}}$, let
$$\eps'(s)=\begin{cases}
  1-\eps(s) & \mbox{ if $s<r$},\\
  \eps(s) & \mbox{ if $s\geq r$}.
\end{cases}$$
Clearly, the mapping $\eps\mapsto \eps'$ is a bijection of $\calE$. Therefore, \eqref{tmp**} follows from
\begin{equation}\label{tmp***}
\sum_{\eps}\left(\PHI\left(\XX_{J(1)}\XX_{J(2)}\prod_{s=3}^{m+2}\ZZ_{J(s)}^{\eps(s)}\right)
  =\sum_{\eps}\PHI\left(\XX_{J(1)}\YY_{J(2)}\prod_{s=3}^{m+2}\ZZ_{J(s)}^{\eps'(s)}\right)\right).
\end{equation}
The latter holds true because by Lemma \ref{L4.5}, for a fixed $\eps$, the left hand side of \eqref{tmp***} is
\begin{equation*}
  \PHI\left(\XX_{J(1)}^2\right)\TAU\left(\XX_{J(2)}\prod_{s=3}^{r-1}\ZZ_{J(s)}^{\eps(s)}\right)
  \PHI\left(\prod_{s=r+1}^{m+2}\ZZ_{J(s)}^{\eps(s)}\right),
\end{equation*}
while the right hand side of \eqref{tmp***} is
$$
\PHI\left(\XX_{J(1)}^2\right)\TAU\left(\YY_{J(2)}\prod_{s=3}^{r-1}\ZZ_{J(s)}^{\eps'(s)}\right)
  \PHI\left(\prod_{s=r+1}^{m+2}\ZZ_{J(s)}^{\eps(s)}\right).
$$
The two expressions are equal, because the joint (mixed) $\TAU$-moments of $\XX_1,\XX_2,\dots, \YY_1,\YY_2,\dots$
by construction do not change when we swap the roles of the
sequences $\{\XX_j\}$ and $\{\YY_j\}$. Of course, such a transformation converts
$\TAU\left(\XX_{J(2)}\prod_{s=3}^{r-1}\ZZ_{J(s)}^{\eps(s)}\right)$   into
 $\TAU\left(\YY_{J(2)}\prod_{s=3}^{r-1}\ZZ_{J(s)}^{\eps'(s)}\right)$.
\end{proof}

\begin{proof}[Proof of Theorem \ref{T2}]
Since convergence of moments is a metric convergence,
we use the standard lemma: to show convergence it suffices to show that every subsequence has a subsequence that converges to the same limit.

The joint $\TAU$-moments of $\UU_n,\VV_n,\SSS_n$ converge,
as the $\TAU$-moments of $\UU_n$ converge by assumption and \eqref{S-limit},  and $\UU_n,\VV_n$ are $\TAU$-free so their joint $\TAU$-moments are
uniquely determined from the moments of $\UU_n$ alone.

By Lemma \ref{L5}, from any subsequence $\UU_{n_k}$ by diagonal method
we can extract a further sub-subsequence such that
the joint $\PHI$-moments of $\UU_n$, $\VV_n$, and $\SSS_n$ converge along that
sub-subsequence. Taken together, the limits of these $\TAU$-moments and $\PHI$-moments
  define a pair of states on
 $\sC\langle\UU,\VV\rangle$, which we will denote again
 by $\TAU$ and $\PHI$. Since $\UU_n,\VV_n$ are $\TAU$-free and $(\PHI|\TAU)$-free
under the limit state $\UU,\VV$ are also
 $\TAU$-free and $(\PHI|\TAU)$-free. 
 From Lemma \ref{L6}, we see that the pair
$$\XX:=\UU/S,\; \YY:=\VV/S$$
 satisfies the assumptions of Theorem \ref{T1} with $a=b=0$.
 By Proposition \ref{L3}, this determines the $\PHI$-law of
 $\UU$ uniquely.
Therefore, the original sequence $\{\UU_n\}$ converges in $\PHI$-moments to $\UU$, and
the $\PHI$-law of $\frac1{\sqrt{S_1^2+\dots+S_n^2}}\sum_{j=1}^n\XX_j=
\sqrt{\frac{n}{S_1^2+\dots+S_n^2}}\UU_n$ converges in $\PHI$-moments to $\UU/S$.
Since the $\TAU$-law of $\UU/S$ is the dilations by $s/S$ of measure $
\nu$, we get formula \eqref{G-x}.
 \end{proof}
\begin{proof}[Proof of Corollary \ref{C-kargin}]
By Ref.  \cite{Kargin-07}, or by repeating the proof of Theorem \ref{T2}
in the special case when $\PHI=\TAU$ with Ref.  \cite{Bozejko-Bryc-04} used instead of Theorem \ref{T1},
we know that the $\TAU$-moments of
$\frac1{\sqrt{S_1^2+\dots+S_n^2}}\sum_{j=1}^n\XX_j$ converge to the semicircle law of variance $\sigma^2=s^2/S^2$.

Since the semicircle law has compact support, we can use Theorem \ref{T2};
the limiting distribution
is then given by \eqref{G_X}.
\end{proof}

\subsection*{Acknowledgements}
This research was partially supported by  the Taft Research Center,  KBN Grant No 1 PO3A 01330, and   NSF
grant \#DMS-0504198.
The second named author thanks Magda Peligrad for bringing Ref.  \cite{Kargin-07} to his attention. The paper benefited from comments by Luigi  Accardi.

\end{document}